\documentclass[10pt,a4paper,titlepage]{article}

\usepackage{amsmath, amssymb, amsfonts, amsthm, float, epsfig, stmaryrd}
\usepackage[latin1]{inputenc}
\usepackage[ngerman, english]{babel}
\usepackage{graphicx, color}
\usepackage[ps,all,arc,rotate]{xy}

\usepackage[pdfpagelabels, pdfborder={0 0 0 [0 0 ]}]{hyperref}

\usepackage{fancyhdr}
\author{Thomas Haettel, Dawid Kielak, Petra Schwer}
\title{The $6$-strand braid group is CAT(0)}

\makeatletter
\def\imod#1{\allowbreak\mkern10mu({\operator@font mod}\,\,#1)}
\makeatother
\numberwithin{figure}{section}

\theoremstyle{plain}
\newtheorem{thm}{Theorem}[section]
\newtheorem*{prop*}{Proposition}
\newtheorem*{thm*}{Theorem}
\newtheorem*{thmA}{Theorem \ref{thm: NCP6 is CAT(1)}}
\newtheorem*{thmB}{Theorem \ref{thm:modular are cat0}}
\newtheorem*{corD}{Corollary \ref{cor: braid groups are CAT0}}
\newtheorem{prop}[thm]{Proposition}
\newtheorem{lem}[thm]{Lemma}
\newtheorem{cor}[thm]{Corollary}
\theoremstyle{definition}
\newtheorem{ex}[thm]{Example}

\newtheorem{dfn}[thm]{Definition}
\newtheorem*{dfn*}{Definition}
\newtheorem{notation}[thm]{Notation}

\theoremstyle{remark}
\newtheorem{rmk}[thm]{Remark}

\def\C{\mathbb{C}}
\def\R{\mathbb{R}}

\def\Z{\mathbb{Z}}
\def\F{\mathbb{F}}

\def\iff{if and only if }

\def\rk{\mathrm{rk}}
\def\crk{\mathrm{crk}}
\def\ra{\rightarrow}
\def\bs{\backslash}
\def\llb{\llbracket}
\def\rrb{\rrbracket}
\def\lk{\mathrm{lk}}

\newcounter{dawidcomments}

\newcounter{thomascomments}

\newcounter{petracomments}

\begin{document}
\textsc{\begin{LARGE}\begin{center} The $6$-strand braid group is CAT(0) \end{center}\end{LARGE}}

\medskip

\begin{center}
Thomas Haettel\footnotemark, Dawid Kielak\footnotemark, Petra Schwer\footnotemark \newline
\today
\end{center}

\footnotetext[1]{Institut Montpelliérain Alexander Grothendieck, Universit\'e de Montpellier, Place Eug\`{e}ne Bataillon, 34095 Montpellier Cedex 5, France, thomas.haettel@univ-montp2.fr}
\footnotetext[2]{Mathematisches Institut, Universit\"{a}t Bonn, Endenicher Allee 60, 53115 Bonn, Deutschland, kielak@math.uni-bonn.de, author supported by the Foundation for Polish Science and ERC grant nb. 10160104}
\footnotetext[3]{Department of Mathematics, Karlsruhe Institute of Technology, Englerstrasse 2, 76133 Kalrsruhe, petra.schwer@kit.edu, author supported by the Foundation for Polish Science and DFG grant SCHW 1550-01}

\medskip

\begin{center}
\begin{minipage}{0.8\textwidth}
\textsc{Abstract.} We show that braid groups with at most $6$ strands are CAT(0) using the close connection between these groups, the associated non-crossing partition complexes,  and the embeddability of their diagonal links into spherical buildings of type~A. Furthermore, we prove that the orthoscheme complex of any bounded graded modular complemented lattice is CAT(0),  giving a partial answer to a conjecture of Brady and McCammond.
\end{minipage}
\end{center}

\bigskip

\section{Introduction}

A discrete group is called CAT(0) if it acts properly discontinuously and cocompactly by isometries on a CAT(0) space. The property of being CAT(0) has far reaching consequences for a group. Algorithmically, such groups have quadratic Dehn functions and hence soluble world problem; geometrically, all free-abelian subgroups are undistorted; algebraically, the centralisers of infinite cyclic subgroups split.

In \cite{charney} Charney asked whether all braid groups are CAT(0). In this paper we give a positive answer for braid groups with at most 6 strands.

Brady and McCammond showed in \cite{brady_mccammond} that the $n$-strand braid groups are CAT(0) if $n=4$ or $5$. However, their proof for $n=5$ relies heavily on a computer program. They also conjectured that the same statement should hold for arbitrary $n$ \cite[Conjecture 8.4]{brady_mccammond}.

This paper exploits the close relationship between braid groups, non-crossing partitions of a regular $n$-gon, and the geometry of spherical buildings; the latter relationship was discovered by Brady and McCammond~\cite{brady_mccammond}.

More specifically, we look at the orthoscheme complex (a certain metric polyhedral complex)  associated to $NCP_n$, the lattice of non-crossing partitions, whose geometry was studied in \cite{brady_mccammond}. Brady and McCammond showed that the CAT(0) property for braid groups can be deduced from the fact that the orthoscheme complex of the non-crossing partition lattice $NCP_n$ is a CAT(0) space. This can be done by inspecting the diagonal link of the orthoscheme complex of $NCP_n$ and proving that this diagonal link is CAT(1).

The diagonal link of the orthoscheme complex of the lattice $NCP_n$ can be embedded into a spherical building of type $A_{n-2}$. Our approach is based on investigating the relationship between the geometry of the diagonal link of $NCP_n$ and the ambient building. Following the criterion of Gromov (see~\cite{gromov}), made precise by Bowditch (see~\cite{bowditch}) and Charney--Davis (see~\cite{charney_davis_singular}), the result will be implied by two facts: it is enough to show that the diagonal link is locally CAT(1) and that it does not contain any unshrinkable locally geodesic loop of length smaller than $2\pi$. We follow this strategy in the proof of the following theorem.

\begin{thmA}
For every $n \leq 6$  the diagonal link in the orthoscheme complex of non-crossing partitions $NCP_n$ is CAT(1).
\end{thmA}

As a consequence we obtain
\begin{corD}
For every $n \leq 6$, the $n$-strand braid group is CAT(0).
\end{corD}

We are thus giving a new proof of the theorem in case $n=4$ or $5$, and provide more evidence (with the newly covered case $n=6$) towards~\cite[Conjecture 8.4]{brady_mccammond}. Note that our proof at no point relies on computer-assisted calculations; it is geometric in flavour.

\smallskip

Brady and McCammond conjectured further that the orthoscheme complex of any bounded graded modular lattice is CAT(0) \cite[Conjecture~6.10]{brady_mccammond}. We are able to give a partial result towards the solution of this problem.

\begin{thmB}
The orthoscheme complex of any bounded graded modular complemented lattice is CAT(0).
\end{thmB}

\noindent \textbf{Outline of proof.} To prove our main result Theorem~\ref{thm: NCP6 is CAT(1)} we first embed the diagonal link of the orthoscheme complex of non-crossing partitions $NCP_n$ into a spherical building. Then we assume (for a contradiction) that the diagonal link contains an unshrinkable (and hence locally geodesic) short loop. The image of such a loop $l$ contains a positive finite number of points of special interest (called turning points) which characterise the positions at which the loop fails to be locally geodesic in the ambient space.

By inspection we show that there is a short path $p$ between any turning point of $l$ and the point opposite to it in $l$, such that $p$ passes through a special type of a vertex, called universal. We show that any short loop passing through a universal vertex is shrinkable, and thus the
two new loops obtained by following half of $l$ and then $p$ are short and shrinkable.
Then a result of Bowditch~\cite{bowditch} concludes the argument.

\smallskip
\noindent \textbf{Other Artin groups.}
In fact Charney~\cite{charney}
stated a more general question about the curvature of Artin groups, and suggested that all of them should be CAT(0). Several partial answers to this question are known. Brady and McCammond studied new presentations for certain three-generator Artin groups \cite{brady_mccammond2} and showed that the associated presentation 2-complex admits a metric of non-positive curvature.

Charney and Davis~\cite{charney_davis} introduced the Salvetti complex, a piecewise Euclidean cube complex, associated to an Artin group. They showed that its universal cover, on which the Artin group acts geometrically, is CAT(0) if and only if the Artin group is right-angled (i.e. the exponents appearing in the presentations are either equal to 2 or $\infty$).

Brady~\cite{brady} studied a class of Artin groups with three generators and constructed certain complexes using the associated Coxeter groups. He showed that these complexes carry a piecewise Euclidean metric of non-positive curvature and have as fundamental group the Artin groups under consideration. A generalisation of this was proved by Bell~\cite{bell}.

Explicit examples of Artin groups with two-dimensional Eilenberg-McLane spaces which act geometrically on 3-dimensional CAT(0) complexes (but not so on 2-dimensional ones) were constructed by Brady and Crisp~\cite{brady_crisp}.

\smallskip
\noindent \textbf{Acknowledgements.} We would like to thank Piotr Przytycki for introducing us to the problem, and all the help provided during our work. We would also like to thank Brian Bowditch who helped us access his `Notes on locally CAT(1) spaces'~\cite{bowditch}. Finally, we would like to thank Ursula Hamenst\"{a}dt for useful conversations, and the referee, for pointing out numerous ways of improving the presentation of this paper.

\section{Definitions and preliminaries}\label{sec: definitions}

We use this section to collect definitions of the main objects of the paper as well as their most important properties. References are given for further reading as well as for all properties listed.

\subsection{Posets and lattices}
This first subsection is used to introduce partially ordered sets (posets) and the geometric realization considered in this paper.

\begin{dfn}[Intervals of integers]
We will use $\llb n,m \rrb$ to denote the interval in $\Z$ between $n$ and $m$ (with $n \leqslant m$), that is
\[ \llb n,m \rrb = [n,m] \cap \Z.\]
\end{dfn}

\begin{dfn}[Posets]
A \emph{poset} $P$ is a set with a partial order. $P$ is \emph{bounded} if it has a (unique) minimum, denoted by $0$, and a (unique) maximum, denoted by $1$. That is for every $x \in P$ we have
\[0 \leq x \leq 1 \]
\end{dfn}

\begin{dfn}[Subposets]
\label{dfn: subposet}
Let $P \subseteq Q$ be two posets. We say that $P$ is a \emph{subposet} of $Q$ \iff the order on $P$ is induced by the order on $Q$.
\end{dfn}

\begin{dfn}[Rank]
A bounded poset $P$ has \emph{rank} $n$ \iff every chain is contained in a maximal chain with $n+1$ elements. For $x \leq y$ in $P$, the \emph{interval} between $x$ and $y$ is the subposet $P_{xy}=\{z \in P \,:\, x \leq z \leq y\}$. If every interval in $P$ has a rank, $P$ is \emph{graded}. If $x$ is an element in a bounded graded poset $P$, then the \emph{rank} of $x$ is the rank of the interval $P_{0x}$.
\end{dfn}

\begin{dfn}[Joins, meets, lattices]
A poset $P$ is called a \emph{lattice} \iff for every $x,y \in P$ the following two conditions are satisfied:
\begin{itemize}
 \item there exists a unique minimal element $x \vee y$ of the set
  \[\{z \in P \mid x \leq z \mbox{ and } y \leq z\},\] called the \emph{join} of $x$ and $y$, and
 \item there exists a unique maximal element $x \wedge y$ of the set \[\{z \in P \mid x \geq z \mbox{ and } y \geq z\},\] called the \emph{meet} of $x$ and $y$.
\end{itemize}
\end{dfn}

\begin{dfn}[Linear lattice]
\label{dfn: linear lattice}
If $V$ is an $n$-dimensional vector space over a division algebra, we will denote by $S(V)$ the rank $n$ lattice consisting of all vector subspaces of $V$, with the order given by inclusion. We call $S(V)$ the \emph{linear lattice} of $V$.
\end{dfn}

It is easy to see that $S(V)$ is indeed a lattice, where the meet of two linear subspaces can be taken to be their intersection and the join is given by their common span.

\begin{dfn}[Failing modularity]
\label{dfn: failing}
Let $P$ be a subposet of a linear lattice $S(V)$. We say that two elements $x,y \in P$ \emph{fail modularity} (with respect to $P$) \iff their join or their meet in $S(V)$ is not contained in $P$.

\end{dfn}

When it is clear from the context in which pair $P \subseteq S(V)$ we are working we sometimes just say $x$ and $y$ fail modularity.

\begin{dfn}[Realisations]
Let $P$ be a graded poset. The \emph{simplicial realisation} $\|P\|$ of $P$ is the simplicial complex whose vertex set is $P$, and whose $k$-simplices correspond to chains $x_0 < x_1 < \ldots < x_k$ of length $k$.

A \emph{geometric realisation} of $P$ is a metric space $X$ together with a homeomorphism $X \to \| P \| $.
\end{dfn}
Note that in particular $\| P \| $ endowed with the standard piecewise-Euclidean metric is a geometric realisation of $P$. This is however \emph{not} the metric we will study in this paper; the metric of our interest will be defined in Definition~\ref{def:orthoscheme}.

Observe that for a bounded poset $P$ the edge connecting $0$ to $1$ is contained in every maximal simplex.

\begin{notation}[Geometric realisation]
Given a (fixed) geometric realisation of $P$, we can (and will) treat $X$ as a simplicial complex via the given homeomorphism $X \to \| P \|$.  Thus we will continue to use the simplicial complex vocabulary when talking about $X$ and we will write $\vert P\vert$ to denote $X$ together with its simplicial structure inherited from $\| P\|$.
\end{notation}

We will also use some standard buildings terminology in the setting of simplicial complexes:
\begin{notation}
In a simplicial complex a \emph{vertex} is a  0-simplex, \emph{faces} are simplices and \emph{chambers} are maximal simplices.
\end{notation}

As we will never look at more than one geometric realisation of any poset, we sometimes abuse notation by using $P$ to denote both a poset and some fixed geometric realisation thereof.

\begin{dfn}[Adjacency]
Given a poset $P$, an element $y \in P$ is said to be \emph{adjacent} to a chain $x_0 < \dots < x_k$ in $P$ \iff $y$ does not belong to the chain, and there exists a chain containing both $y$ and all the elements $x_i$.

In the setting of the simplicial realisation $\| P \|$, a vertex $y$ is adjacent to a face $F = \{x_0, \dots, x_k\}$ \iff $y$ does not belong to $F$, and $F \cup \{y\}$ is itself a face. Equivalently, $y$ is adjacent to $F$ \iff for each vertex $x_i \in F$ there is an edge connecting $x_i$ to $y$.
\end{dfn}

\begin{dfn}[Diagonal link]
Given a geometric realisation $\vert P \vert $ of a bounded lattice $P$, we define the \emph{diagonal link} of  $\vert P \vert $ to be the link
\[ LK(e_{01}, \vert P \vert) \]
of the \emph{diagonal edge} $e_{01}$, that is the edge connecting the minimum 0 to the maximum 1.
\end{dfn}

Note that if $|P|$ has a piecewise Euclidean or spherical metric, then $LK(e_{01},\vert P \vert)$ carries a natural angular metric and is hence itself a geometric realisation of the poset $P \setminus \{0,1\}$.

\begin{rmk}\label{rem: labels in link}
A fact we will frequently use is that the vertices of $LK(e_{01},\vert P \vert)$ are in a natural way in one to one correspondence with elements of $P\setminus \{0,1\}$ as follows. A vertex, i.e. 0-simplex in $LK(e_{01},\vert P \vert)$ corresponds to a 2-simplex in $\vert P\vert$ whose vertices are $0,1$ and one additional vertex $p\in P$. We may thus label the corresponding 0-simplex in the link by $p$. So when we refer to a vertex $p$ in $LK(e_{01},\vert P \vert)$ what we mean is the 0-simplex in the link which corresponds to the 2-simplex in $\vert P\vert$ spanned by $p$ and $0,1$.
\end{rmk}

\subsection{Spherical buildings and orthoscheme complexes}

First we will very quickly introduce spherical buildings and some of their basic properties. In the rest of this section we focus on the spherical buildings of type $A_n$, which arise from the lattice of linear subspaces of a vector space. In the rest of the paper we will use the standard terminology of spherical buildings freely and refer the reader to the book by Abramenko and Brown~\cite{abramenko_brown} for further details.

\begin{dfn}
A (spherical) \emph{building} is a simplicial complex $B$ which is the union of a collection of subcomplexes $A$, called \emph{apartments}, satisfying the following axioms:
\begin{itemize}
\item[(B0)] Each apartment is isomorphic to a (finite) Coxeter complex.
\item[(B1)] For any two simplices $c,d$ in $B$ there exists an apartment containing both.
\item[(B2)] If $A_1$ and $A_2$ are two apartments containing simplices $c$ and $d$, then there exists an isomorphism $A_1\to A_2$ fixing $c$ and $d$ pointwise.
 \end{itemize}
 The maximal simplices in $B$ are called \emph{chambers}.
\end{dfn}

Note that $c,d$ are allowed to be empty in axiom (B2) and hence any two apartments are isomorphic. We call the type of the Coxeter group the \emph{type} of the building. Note further that $B$ is a chamber complex, that is any two maximal simplices have the same dimension.

For any spherical building $B$ there is a standard geometric realisation of $B$ which induces on each apartment the round metric of a sphere. Throughout the paper we will consider a spherical building $B$ simultaneously as a simplicial complex and a metric space using this standard geometric realisation.

\begin{dfn}
Two points $x,y$ in a spherical building $X$ are called \emph{opposite} if for some (any) apartment $A$ containing $x$ and $y$, $x$ and $y$ are opposite in the apartment $A$, seen as a round sphere. Equivalently, the distance between $x$ and $y$ in $X$ is $\pi$. Two faces $F$,$F'$ in $X$ are called \emph{opposite} if for some (any) apartment $A$ containing $F$ and $F'$, $F$ and $F'$ are opposite in the apartment $A$.
\end{dfn}

\begin{prop}\label{prop:buildings}
 Let $B$ be a spherical building. Then
\begin{enumerate}
\item\label{one} the link $\lk_B(c)$ of any simplex $c$ in $B$ is a spherical building.
\item\label{two} apartments are metrically convex, in other words for any apartment $A$ containing a pair of points $x,y\in B$ one has $d_B(x,y)=d_A(x,y)$.
\item\label{three} $B$ is a CAT(1) space when equipped with the standard metric.
\end{enumerate}
\end{prop}
\begin{proof}
For a proof of item~\ref{one} see Proposition 3 on page 79 in the book of Brown~\cite{brown}. Proofs of the other items are contained in~the book of Bridson--Haefliger~\cite{bridson_haefliger}. Item~\ref{two} follows from Lemma II.10A.5  and item~\ref{three} is Theorem II.10A.4 therein.\end{proof}

We are interested in one particular type of buildings namely the spherical buildings of type $A_n$.
To see what they are recall that if $V$ is an $n$-dimensional vector space over a division algebra, the linear lattice $S(V)$ of $V$ is the rank $n$ lattice consisting of all vector subspaces of $V$, with the order given by inclusion.
One can equip the simplicial realization of a building with a so called orthoscheme metric that will allow us to explicitly describe the standard CAT(1) metric on buildings of type $A$.

\begin{dfn}[Orthoscheme complex] \label{def:orthoscheme}
Let $P$ be a bounded graded poset. A maximal chain $x_0 < \dots < x_n$  in $P$ corresponds to an $n$-simplex $F$ in the simplicial realisation $\| P \|$ of $P$. We endow this simplex with a metric in the following way: we identify each $x_i$ with the vertex $(1,\ldots,1,0,\ldots,0)$ ($i$ times ``1'') in $\R^n$. We give $F$ the metric of the Euclidean convex hull of the vertices in the cube. Equivalently, it is the metric on simplices of the barycentric subdivision of the Euclidean $n$-cube with side length $2$.

Note that the distance between two vertices lying in a common simplex depends only on the difference in their rank. We can endow each maximal simplex (i.e. chamber) in $\| P \|$ with this metric in a coherent way. The induced length metric on the whole complex is the \emph{orthoscheme metric}. This way $\| P \|$ becomes a geometric realisation of $P$, which is called the \emph{orthoscheme complex} of $P$.
\end{dfn}

For more information about the orthoscheme complexes we refer the reader to Brady and McCammonds exposition in ~\cite{brady_mccammond}. Brady--McCammond~\cite{brady_mccammond} show the following.

\begin{prop}\label{prop:building type A}
When $V$ is an $n$-dimensional vector space over a division algebra, then the diagonal link $LK(e_{01}, \vert S(V) \vert)$ of the orthoscheme complex $\vert S(V) \vert$ is equal to the (standard CAT(1) realisation of the) spherical building associated to $\mathrm{PGL}(V)$, which is a spherical building of type $A_{n-1}$. The dimension of this building is $n-2$.
\end{prop}

The above proposition is crucial for the proof of our main result. It is precisely the geometry of this building and (a specific class of) its subcomplexes that we will focus on.

One can show that apartments of $B$, the building associated to $\mathrm{PGL}(V)$, are in one-to one correspondence with bases of $V$.
\smallskip

Note that, according to Remark~\ref{rem: labels in link}, the vertex set of $LK(e_{01}, \vert S(V) \vert)$ has the structure of a bounded graded lattice if only we add to it the minimum 0 (corresponding to the trivial subspace) and the maximum 1 (corresponding to the improper subspace).
Because of this deficiency let us use the following convention.

\begin{dfn}
We say that a subset $M$ of the vertex set of the diagonal link $LK(e_{01}, \vert S(V) \vert)$ is \emph{stable under joins and meets} \iff  for every $x,y \in M$ we have $x \vee y \in M \cup \{1\}$ and $x \wedge y \in M \cup \{ 0 \}$.
\end{dfn}

\begin{dfn}
Let $M \subseteq B$ be a subset of a building. The \emph{simplicial convex hull} of $M$ is defined to be the intersection of all apartments of $B$ containing $M$.
\end{dfn}

\begin{lem}
\label{lem:geodesic in building}
Let $V$ be an $n$-dimensional vector space over a division algebra, and write $B=LK(e_{01},\vert S(V)\vert)$ for the diagonal link of the orthoscheme complex $\vert S(V)\vert$ of $S(V)$ (and hence a spherical building of type $A_{n-1}$). Let $F$ and $F'$ be two faces in $B$. Consider the minimal set $M$ of vertices of $B$ containing the vertices of $F$ and $F'$ which is stable under joins and meets. Then the full subcomplex spanned by $M$ is equal to the simplicial convex hull of $F \cup F'$.
\end{lem}
\begin{proof}
Let $H$ denote the simplicial convex hull of $F \cup F'$ in $B$.
Let $A$ be an apartment in $B$ containing $F \cup F'$, and let $\{e_1,\ldots,e_{n}\}$ be a basis of $V$ corresponding to $A$. As every element of $M$ arises as a joins or meet of vertices of $F$ and $F'$ every element of $M$ is spanned by some elements of $\{e_1,\ldots,e_{n}\}$, so it belongs to the apartment $A$. Since this is true for every apartment containing $F \cup F'$ it holds for their intersections and hence we have proved that $M \subseteq H$.

\smallskip

To show the converse let $F$ and $F'$ be two simplices and suppose there exists some  $v \in H \smallsetminus M$.
We will show that there exists an apartment containing $M$ but not $v$.

Let $A$ be an apartment containing $M$ and let $\{e_1,\ldots,e_{n}\}$ be a basis of $V$ corresponding to $A$.
Suppose that for each $i$ with $\mathrm{e_i} \leqslant v$ there exists $m_i \in M$ with $e_i \leqslant m_i \leqslant v$. Then we have
\[
 v = \bigvee_I e_i \leqslant \bigvee_I m_i \leqslant v
\]
where $I$ denotes the set of $i$ such that $e_i \leqslant v$. But then $v$ is a join of elements in $M$, and thus is itself in $M$ as $M$ is closed under taking joins.
We conclude that
there exists $i \in \llb 1,n \rrb$ such that
$e_i \leqslant v$, and
\[\forall m \in M, e_i \leqslant m \Longrightarrow m \not\leqslant v.\]

Consider $m_0 = \bigwedge \{m \in M \mid e_i \leqslant m\} \in M$, where $m_0 = V$ if there is no $m \in M$ such that $e_i \leqslant m$. Since $e_i \leqslant m_0$, we know that $m_0 \not\leqslant v$, so there exists $j \in \llb 1,n \rrb \bs \{i\}$ such that $e_j \leqslant m_0$ and $e_j \not\leqslant v$. Then the apartment $A'$ corresponding to the basis $\{e_1,\ldots,e_{i-1},e_i+e_j,e_{i+1},\ldots,e_{n}\}$ contains $M$ but not $v$. But then $v\notin H$, which is a contradiction.  So we have proved that $H \subseteq M$.
\end{proof}

\subsection{Non-crossing partitions}

Let us now introduce the pivotal objects in this article, non-crossing partitions. We will see that they form a sublattice of the linear subspace lattice of a vector space.

\begin{dfn}[Partition lattice]
Let $U_n$ be the set of $n^{th}$ roots of unity inside the plane $\C$. A \emph{partition} of $U_n$ is a decomposition of the set $U_n$ into disjoint subsets, called \emph{blocks}, such that their union is $U_n$.
Let $P_n$ denote the set of all partitions of the set $U_n$. The set $P_n$ forms a bounded graded lattice of rank $n-1$, where the order is given by: $p \leq p'$ if and only if every block of $p$ is contained in a block of $p'$.
\end{dfn}

\begin{dfn}[Non-crossing partition lattice]
\label{dfn: ncp type A}
\label{dfn: ncp}
A partition of $U_n$ is called \emph{non-crossing} if for every distinct blocks $x,y$ of the partition, the convex hulls $Hull(x)$ and $Hull(y)$ in $\C$ do not intersect. We define $NCP_n$ to be the subposet of $P_n$ consisting of non-crossing partitions of $U_n$. Then $NCP_n$ is a bounded graded lattice of rank $n-1$.
\end{dfn}

\begin{lem}[NCP is a subposet of $S(V)$]
\label{lem: NCPs are linear}
For every $n \geq 2$, $P_n$ and $NCP_n$ are isomorphic to subposets of $S(V)$, where $V$ is an $(n-1)$-dimensional vector space.
\begin{proof}
Fix a field $\F$, and let $V=\{(y_i) \in \F^n \mid \sum_{i=1}^n y_i =0\}$. Then $V$ is an $(n-1)$-dimensional $\F$-vector space. Identify $U_n$ with $\llb 1,n \rrb$.
If $x \in P_n$ let \[f(x) =  \big\{(y_i) \in V \mid \forall \text{ block } Q \in x : \sum_{i \in Q} y_i=0\big\}.\] Then $f$ is an injective rank-preserving poset map from $P_n$ to $S(V)$. It clearly restricts to $NCP_n \subseteq P_n$.
\end{proof}
\end{lem}

\begin{dfn}[Non-crossing partition complex]
We will refer to the orthoscheme complexes of the non-crossing partition lattices $NCP_n$  as the \emph{non-crossing partition complexes}. It is the simplicial realization of $NCP_n$ equipped with the orthoscheme metric, as defined in Definition~\ref{def:orthoscheme}.
\end{dfn}

\begin{lem}[Duality]
\label{lem: duality}
For $n \geq 2$, there is a duality on $NCP_n$, i.e. an order-reversing bijection $x \mapsto x^*$ from $NCP_n$ to itself.
\begin{proof}
Denote by $\{\omega_k\}_{k \in \Z/n\Z} = U_n$ the $n^{th}$ roots of unity. If $x$ is a non-crossing partition of $U_n$, then its dual $x^*$ is the partition of the shifted set
\[U_{n}^*=\{m_{k}=e^{\frac{\pi i}{n}} \omega_k \}_{k \in \Z/n\Z}\]
 defined by: $m_{k}$ and $m_{j}$ belong to the same block of $x^*$ \iff the geodesic segment $[m_{k},m_{j}]$ in $\C$ does not intersect the convex hull of any block of $x$. Then $x^*$ is a non-crossing partition of $U_{n}^*$, with $\rk (x^*) = n-1-\rk (x)$, and ${(x^*)}^*=x$. Now choose some identification between $U_{n}^*$ and $U_n$ (like multiplying by $e^{\frac{-\pi i}{n}}$) to get a map from $NCP_n$ to itself.
\end{proof}
\end{lem}

Note that we will only use duality to reduce the number of cases that will need checking in the later stage of our proof.

\subsection{CAT(0) and CAT(1) spaces}

In this section we will state the definitions of CAT(0) and CAT(1) spaces and recall some of Bowditch's results about locally CAT(1) spaces (see~\cite{bowditch}). Moreover we recall how Brady and McCammond use Bowditch's criteria to give a sufficient condition for braid groups to be CAT(0) (see~\cite{brady_mccammond}).
For a general discussion of CAT($\kappa$) spaces we refer the reader to the book by Bridson and Haefliger~\cite{bridson_haefliger}.

\begin{dfn}
Let $X$ be a geodesic metric space. A \emph{geodesic triangle} $\Delta$ is formed by three geodesic segments, $\gamma_i \colon [0, l_i] \to X$ with $i \in \Z / 3\Z$, such that $\gamma_i(l_i) = \gamma_{i+1}(0)$.

Given such a triangle, we form the \emph{Euclidean comparison triangle} $\Delta' \subset \R^2$ by taking any triangle whose vertices $x_1, x_2, x_3$ satisfy
\[ d_{\R^2}(x_i, x_{i+1}) = l_i .\]
There is an obvious map $c \colon \Delta \to \Delta'$, isometric on edges, sending $\gamma_i(0) \mapsto x_i$; we will refer to it as a \emph{comparison map}.

We say that $X$ is \emph{ CAT(0)} \iff for any two points $x,y$ on any geodesic triangle $\Delta$, we have
\[ d_X(x,y) \leq d_{\R^2}(c(x), c(y)).\]
\end{dfn}

\begin{dfn}
Given a geodesic triangle $\Delta$ with notation as above, with the additional condition that $l_1 + l_2 + l_3 \leq 2 \pi$, we form the \emph{spherical comparison triangle} $\Delta'' \subset S^2$ by taking any triangle whose vertices $x_1, x_2, x_3$ satisfy
\[ d_{S^2}(x_i, x_{i+1}) = l_i .\]
Again there is an obvious map $c \colon \Delta \to \Delta''$, isometric on edges, sending $\gamma_i(0) \mapsto x_i$; we will refer to it as a \emph{comparison map}.

We say that $X$ is \emph{CAT(1)} \iff for any two points $x,y$ on any geodesic triangle $\Delta$ with perimeter at most $2 \pi$, we have
\[ d_X(x,y) \leq d_{S^2}(c(x), c(y))\]
\end{dfn}

\begin{dfn}
A group $G$ has the \emph{CAT(0) property}, or \emph{is CAT(0)},  if and only if it acts properly discontinuously and cocompactly by isometries on a CAT(0) space.
\end{dfn}

\begin{dfn}[Locally CAT(1)]
\label{dfn: locally cat1}
A complete, locally compact, path-metric space $X$ is said to be \emph{locally} CAT(1) if each point of $X$ has a CAT(1) neighbourhood.
\end{dfn}

\begin{dfn}[Shrinking and shrinkable loops]
\label{dfn: shrinkable loops}
Let $X$ be a complete, locally compact path-metric space. A rectifiable loop $l$ in $X$ is said to be \emph{shrinkable to} $l'$ \iff $l'$ is another rectifiable loop in $X$, and there exists a homotopy between $l$ and $l'$ going through rectifiable loops of non-increasing lengths.

A rectifiable loop $l$ is \emph{shrinkable} \iff it is shrinkable to a constant loop.

The loop $l$ is said to be \emph{short} if its length is smaller than $2\pi$.
\end{dfn}

\begin{thm}[{Bowditch~\cite[Theorem~3.1.2]{bowditch}}]
 \label{thm: bowditch criterion}
Let $X$ be a locally CAT(1) space. Then $X$ is CAT(1) if and only if every short loop is shrinkable.
\end{thm}

The following theorem will be an important tool in our argument.

\begin{thm}[{Bowditch~\cite[Theorem~3.1.1]{bowditch}}]
 \label{thm: bowditch shortcut}
Let $X$ be a locally CAT(1) space. Let $x,y \in X$, and consider three paths $\alpha_1,\alpha_2,\alpha_3 \colon [0,1] \ra X$ joining $x$ to $y$. For all $i \in \{1,2,3\}$, consider the loop $\gamma_i =  \alpha_{i+1}^{-1} \circ \alpha_i$ based at $x$ (with indices modulo 3). Assume that for all $i \in \{1,2,3\}$ the loop $\gamma_i$ is short. Assume further that $\gamma_1$ and $\gamma_2$ are shrinkable. Then $\gamma_3$ is shrinkable.
\end{thm}

\begin{thm}[{Brady, McCammond~\cite[Theorem~5.10 and Lemma~5.8]{brady_mccammond}}]
 \label{thm: brady mccammond large}
Assume that for all $3 \leq k \leq n$, the diagonal link of $\vert NCP_k\vert$ does not contain any unshrinkable short loop. Then $\vert NCP_n\vert$ is CAT(0).
\end{thm}

\begin{prop}[{Brady, McCammond~\cite[Proposition~8.3]{brady_mccammond}}]
 \label{prop: cat0 implies braids}
If $\vert NCP_m\vert$ is CAT(0) for all $m \leqslant n$, then the $n$-strand braid group is CAT(0), that is it acts geometrically on a CAT(0) space.
\end{prop}

\section{Turning points and turning faces}

In order to get ready for the proof of our main result Theorem~\ref{thm: NCP6 is CAT(1)}, we introduce further tools: turning points and turning faces. Some of their properties hold more generally in arbitrary path-connected subset of metric spaces equipped with the induced length metric.  We collect definitions and properties in the following two subsections.

\subsection{Turning points}

Turning points are points on locally geodesic loops in a subspace of a metric space where said loop fails to be a local geodesic in the ambient space. The precise definition is a follows.

\begin{dfn}[Turning points]
\label{dfn: turning point}
Let $X$ be a path-connected subspace of a geodesic metric space $B$, and endow $X$ with the induced length metric. Suppose that $l \colon D \to X$ is a local isometry, where $D$ is a metric space. We say that a point $t \in D$ is a \emph{turning point} of $l$ in $B$  \iff $i \circ l$ fails to be a local isometry at $t$, where $i \colon X \to B$ is the inclusion map.
\end{dfn}

\begin{dfn}[Locally geodesic loops]
\label{dfn: locally geodesic loops}
Let $X$ be a metric space.
We say that $l \colon S^1 \to X$ is a \emph{locally geodesic loop} in $X$ \iff $l$ is a local isometry, where $S^1$ is given the length metric of the quotient of some closed interval $I$ of $\R$ by its endpoints.
The \emph{length} of $l$ is defined to be the length of $I$.

We say that $l \colon I \to X$ is a \emph{locally geodesic path} in $X$ \iff $l$ is a local isometry, where $I$ is a closed interval of $\R$. The \emph{length} of $l$ is defined to be the length of $I$.
\end{dfn}

\begin{lem}
\label{lem: convex nhood}
Let $X$ be a path-connected subset of a geodesic metric space $B$, and endow $X$ with the induced length metric. Suppose that $l \colon D \to X$ is a locally geodesic path or loop, with $D$ being respectively $I$ or $S^1$. Let $t \in D$. Suppose that there exists a subset $N \subseteq X$, such that $N$ contains the convex hull
in $B$ of the image under $l$ of some neighbourhood of $t$ in $D$. Then $t$ is not a turning point.
\begin{proof}
Suppose (for a contradiction) that $t$ is a turning point. As $i \circ l$ fails to be a local geodesic at $t$,  there exist $t_1, t_2 \in D$ in a neighbourhood of $t$ such that
\[ d_B\big(l(t_1), l(t_2)\big) < d(t_1,t_2) \]
and such that $l(t_j) \in N$ for $j \in  \{1,2\}$.

Since $B$ is a geodesic metric space, we can realise the distance between $l(t_1)$ and $l(t_2)$ with a geodesic segment $g$ in $B$. Since $N$ contains the endpoints of $g$, it contains the whole of $g$. Hence in particular $g$ lies in $X$, which (as was claimed) contradicts the fact that $l$ was a local geodesic.
\end{proof}
\end{lem}

\begin{rmk}
We will often identify (isometrically) a neighbourhood of a point $t \in S^1$ with an interval in $\R$ containing $t$ in its interior. We will therefore feel free to write $[t -\varepsilon, t + \varepsilon]$ etc. (for a small $\varepsilon$) to denote a subset of $S^1$.
\end{rmk}

\begin{dfn}[Consecutive turning points]
Suppose that we have a subset $T \subset S^1$. We will say that $t, t' \in T$ are \emph{consecutive} \iff there is a path in $S^1$ with endpoints $t$ and $t'$ not containing any other point in $T$. A shortest such path will be denoted by $[t,t']$.
\end{dfn}

\begin{rmk}
Note that $[t,t']$ defined above is unique provided that the  cardinality of $T$ is at least 3.
\end{rmk}

\begin{lem}
\label{lem: cardinality of T}
Let $X$ be a path-connected subset of a CAT(1) space $B$, and endow $X$ with the induced length metric. If $l \colon S^1 \to X$ is a locally geodesic loop in $X$ of length $0<L < 2\pi$, then the cardinality of the set of turning points $T$ of $l$ is greater than 2. Moreover, $l\vert_{[t,t']}$ is a geodesic in $B$ for any pair of consecutive turning points $t,t'$.
\begin{proof}
Suppose that we can find three distinct points $t_1, t_2, t_3 \in S^1$ such that each pairwise distance is strictly bounded above by $\frac1 2 L < \pi$, and such that $T$ is contained in $[t_1,t_2] \cup \{t_3\}$, where $[t_i,t_j]$ denotes the shortest segment of $S^1$ with endpoints $t_i$ and $t_j$ not containing $t_k$ for \[\{i,j,k\} = \llb 1,3 \rrb.\] Note that $l\vert_{[t_1,t_3]}$ and $l\vert_{[t_2,t_3]}$ are geodesics in $B$ -- this follows from the fact that local geodesics of length at most $\pi$ are geodesics in CAT(1) spaces (essentially because the statement is true for the 2-sphere $S^2$).

Consider a geodesic triangle $\Delta = l(t_1) l(t_2) l(t_3)$ in $B$, and let $\Delta'$ be the comparison triangle in $S^2$. Since $L < 2 \pi$, the perimeter of $\Delta$ (and hence also of $\Delta'$) is smaller than $2 \pi$. Therefore $\Delta'$ cannot be a great circle in $S^2$.

Suppose that $t_3 \not\in T$. Then the angle of $\Delta$ at $l(t_3)$ is equal to $\pi$, and the same is true in the comparison triangle $\Delta'$ (by the CAT(1) inequality). But then the triangle is degenerate, and hence so is $\Delta$. In particular the geodesic from $l(t_1)$ to $l(t_2)$ goes via $l(t_3)$. But this contradicts the assumption that the distance (in $B$) between $l(t_1)$ and $l(t_2)$ is smaller than $\frac1 2 L$. So $t_3 \in T$. We will now use this trick to prove our claims.

If $\vert T \vert \leqslant 1$ then we immediately get a contradiction by taking either any three points in $S^1$ satisfying the conditions above, or the turning point and two other points so that the triple satisfies the condition.

If $\vert T \vert = 2$ and the two points are not antipodal in $S^1$, then we can always (very easily indeed) find a third point so that the triple satisfies our condition. If the turning points are antipodal, then the two local geodesics given by $l$, which connect the images of the turning points, coincide. This is because local geodesics of length smaller than $\pi$ are geodesics in $B$, and such geodesics in CAT(1) spaces are unique. But then $l$ cannot be a local geodesic in $X$.
We have thus shown that $\vert T \vert \geqslant 3$.

Now suppose we have two consecutive turning points, $t_1$ and $t_2$. If $[t_1,t_2]$ is of length at most $\pi$, then $l\vert_{[t_1,t_2]}$ is a geodesic as before. If the length is larger than $\pi$, then in particular it is larger than $\frac 1 2 L$, and so we can take the midpoint $t_3 \in [t_1,t_2]$ and (applying the argument above) conclude that $t_3 \in T$, which in turn contradicts the definition of $[t_1,t_2]$.
\end{proof}
\end{lem}

\subsection{Turning faces }

We are mainly interested in locating turning points on loops in linearly embedded subcomplexes of the orthoscheme complex of linear subspaces of a vector space. To understand their behaviour we use properties of the supporting faces which will be called turning faces and are introduced in this section.

\begin{dfn}
\label{dfn: linearly embedded}
Let $V$ be an n-dimensional vector space over a division algebra and denote by $B$ be the diagonal link $LK(e_{01},\vert S(V) \vert)$ of the orthoscheme complex $\vert S(V) \vert$ of $S(V)$. Hence $B$ is a spherical building of type $A_{n-1}$ equipped with the standard CAT(1) metric.
We say that a geometric realization $X$ of a simplicial complex is \emph{linearly embedded in $B$} \iff there exists a bounded graded lattice $P$ of rank $n$, such that
\begin{enumerate}
\item $P$ is a subposet of $S(V)$ with a geometric realisation $\vert P \vert$ isometric to the full subcomplex of $\vert S(V) \vert$ spanned by $P$, and
\item $X$ is isometric (and isomorphic as simplicial complexes) to the diagonal link $LK(e_{01},\vert P \vert) \subseteq B$ equipped with the length metric induced from $B$.
\end{enumerate}
We will call the metric on $X$ the \emph{spherical orthoscheme} metric. Note that $X$ has dimension $n-2$.
\end{dfn}

Since a linearly embedded $X$ is a subcomplex of a building, we will use the standard buildings vocabulary when talking about $X$. Hence  \emph{a chamber} or \emph{an apartment}  in $X$ is a chamber or an apartment of $B$, respectively, which is contained in $X$.
Note also that since the ranks of $P$ and $S(V)$ agree, the complex $X$ is a union of chambers.

For the remainder of this subsection let $X$ be linearly embedded in $B$.

\begin{dfn}[Rank and corank]
Let $F$ be a face of codimension $m$ in $X$. Then $F$ is the span of vertices $x_1, \ldots, x_{n-1-m}$ of ranks $r_1, \ldots, r_{n-1-m}$. We define the set $\rk (F) = \{r_1, \ldots , r_{n-1-m}\}$ to be the \emph{rank} of $F$, and the set $\crk (F) = \llb 1, n-1 \rrb \smallsetminus \rk(F)$ to be the \emph{corank} of $F$.
\end{dfn}

\begin{dfn}[Turning face]
Let $t$ be a turning point of a locally geodesic loop $l \colon S^1 \to X$. The smallest (with respect to inclusion) intersection $F$ of a chamber in $X$ containing $l([t,t-\varepsilon))$ and a chamber in $X$ containing $l([t,t+\varepsilon))$, for sufficiently small $\varepsilon>0$, will be called the \emph{turning face} of $t$.
\end{dfn}

\begin{lem}
\label{lem: finite}
Let $l$ be a locally geodesic loop or path in $X$. Then the set $T$ of turning points of $l$ is finite.
\begin{proof}
Let $l \colon D \ra X$, where $D=S^1$ or $D=I=[0,L]$ and $l(0)\neq l(L)$. If $D=I$, notice that for $\varepsilon>0$ small enough $l([0,\varepsilon])$ is contained in a chamber in $X$, hence by Lemma~\ref{lem: convex nhood} the point $0 \in I$ is not a turning point of $l$, and similarly neither is $L$.

We claim that $T$ is a discrete subset of $D$. Suppose that $t \in T$ is a turning point. Let $F$ be the turning face of $t$. By definition, there exists $\varepsilon>0$ such that $F$ is the intersection of a chamber $C^-$ in $X$ containing $l([t,t-\varepsilon))$ and a chamber $C^+$ in $X$ containing $l([t,t+\varepsilon))$.

Then $l|_{[t-\varepsilon,t]}$ is the geodesic segment from $l(t-\varepsilon)$ to $l(t)$ in $C^- \subset X$. In particular it is also locally geodesic in $B$, so there is no turning point in $(t-\varepsilon,t)$, and similarly in $(t,t+\varepsilon)$ and $C^+$. Therefore $T$ is discrete.

Note that $T$ is closed -- this follows directly from the fact that if $t \in D$ is not a turning point, then $l$ is a geodesic (in $B$) at some open neighbourhood of $t$ in $D$, and so in particular none of the points in this open neighbourhood are turning points themselves. Hence $T$ is closed, and therefore compact since $D$ is compact.
We have thus shown that $T$ is compact and discrete, and so it is finite.
\end{proof}
\end{lem}

The following lemma is the first result which gives us some combinatorial control over the turning points. It is precisely this type of control which will allow us to perform the inspection in the proof of the main theorem.

\begin{lem}
\label{lem: coconsecutive}
Let $l$ be a locally geodesic loop in $X$. Then for every turning point of $l$ in $B$, its turning face has a corank which contains two consecutive integers.
\begin{proof}
Suppose that we have $t \in S^1$, a turning point of $l$,  whose image $x$ under $l$ is contained in the turning face $F$ in $X$. By definition, there exists $\varepsilon > 0$ such that $l((t-\varepsilon,t]) \subseteq C^-$ and $l([t,t+\varepsilon)) \subseteq C^+$, where $C^-$ and $C^+$ are chambers of $X$ such that $F=C^+ \cap C^-$.

Assume that the corank of $F$ does not contain two consecutive integers. Then the sets of vertices of $C^-$ and $C^+$ differ at vertices of ranks
\[1 \leq r_1< \dots <r_k \leq n-2\]
with $\forall \ i \in \llb 1, k-1 \rrb \, : \, r_{i+1}-r_i \geq 2$. Then for every $\varepsilon = (\varepsilon_1,\ldots,\varepsilon_k) \in \{\pm\}^k$, consider the chamber $C^\varepsilon$ spanned by $C^+ \cap C^-$ and, for every $1 \leq i \leq k$, by the vertex of rank $r_i$ in $C^{\varepsilon_i}$. Since all vertices of $C^\varepsilon$ belong to $X$, we know that $C^\varepsilon$ belongs to $X$.

By item~\ref{one} of Proposition~\ref{prop:buildings} the link of $F=C^+ \cap C^-$ in $B$ is itself a spherical building. This is easily seen by taking successive links of single vertices of $F$. Taking the first such link, we either get a building of type $A_{n-2}$ (when the vertex we removed was of rank 1 or $n-1$), or a building of type $A_l \times A_{n-2-l}$ (when the vertex was neither of minimal nor maximal rank). Repeating this process we obtain a spherical building whose type is determined precisely by the structure of the corank of $F$; more specifically, when the corank has no consecutive integers, the building is of type $A_1^k$. Thus the apartments are spherical joins of $k$ copies of the 0-sphere.

Let $N=\bigcup_{\varepsilon \in \{\pm\}^k} C^\varepsilon$. Observe that the image of $N$ in the link of $F$ in $B$ is precisely one of the apartments, and therefore it is convex.
Since the link $\lk(x,N)$ of $x$ in $N$ is isometric to the spherical join $\lk(x,F) * \lk(F,N)$, it is convex in the link $\lk(x,B) \simeq \lk(x,F) * \lk(F,B)$ of $x$ in $B$. For $\delta>0$ small, the $\delta$-ball around $x$ is isometric to the $\delta$-ball around the cone point in the cone over the link of $x$, according to~\cite[Theorem~7.16]{bridson_haefliger}. Since $\lk(x,N)$ is convex in $\lk(x,B)$, we conclude that the $\delta$-ball around $x$ in $N$ is convex in $B$. Since $N$ contains the image under $l$ of some neighbourhood of $t$ in $S^1$, using Lemma~\ref{lem: convex nhood}  we show that $t$ is not a turning point.
\end{proof}
\end{lem}

\begin{lem}
\label{lem: modular do not turn}
Let $l \colon I \ra X$ be a locally geodesic segment in $X$ with a turning point $t$ in $B$. Let $E^+$ (respectively $E^-$) be minimal faces in $X$ containing the image under $l$ of a right (respectively left) $\varepsilon$-neighbourhood of $t$ for some $\varepsilon >0$.
Then
the simplicial convex hull of $E^+ \cup E^-$ is not contained in $X$.
\begin{proof}
Let $N$ denote the simplicial convex hull of ${E^- \cup E^+}$ in $B$. Suppose (for a contradiction) that $N$ is contained in $X$. The subcomplex $N$ is metrically convex (as it is an intersection of apartments, which are metrically convex in $B$), and contains the image under $l$ of some neighbourhood of $t$ in $S^1$. Therefore using Lemma~\ref{lem: convex nhood}, we show that $t$ is not a turning point, and this concludes the proof.
\end{proof}
\end{lem}

\section{Proof of the main theorem}

In this section we will prove our main result.  Let us first fix some notation and recall a few facts.

Denote by $\vert NCP_n\vert$ the orthoscheme complex of the non-crossing partition lattice for some $n\geq 3$, and let $X$ denote the diagonal link $ LK(e_{01}, \vert NCP_n \vert) $ equipped with the spherical orthoscheme metric. Thus $X$ is the geometric realization of an $n-3$-dimensional simplicial complex whose vertices are in one-to-one correspondence with the partitions in $NCP_n\setminus\{0,1\}$, see Remark~\ref{rem: labels in link}. By the rank of $X$ we mean the rank of the poset  $NCP_n\setminus\{1\}$, which is $n-2$.

Recall that $X$ is linearly embedded in a spherical building $B$, which is the diagonal link of the linear subspace lattice of an $(n-1)$-dimensional vector space $V$ by Lemmata~\ref{lem: NCPs are linear} and \ref{prop:building type A}. Note that this building $B$ has type $A_{n-2}$ and is a simplicial complex of dimension $n-3$.

Moreover, $B$ is a CAT(1) metric space, which is why (as Brady and McCammond remarked, see~\cite[Remark~8.5]{brady_mccammond}) the spherical orthoscheme metric on $X$ is a good candidate to be CAT(1) for all $n \geq 3$.

Recall from Remark~\ref{rem: labels in link} that the vertices of $X$ are naturally labeled by elements of $NCP_n$.  We may thus talk about partitions in $X$.

\begin{rmk} \label{rmk: ncp 3 and 4}
Note that for $n=3$, the diagonal link $X$ in $NCP_3$ is the disjoint union of $3$ points, so it is CAT(1). For $n=4$, the diagonal link $X$ in $NCP_4$ is a subgraph of the incidence graph of the Fano plane, so it has combinatorial girth $6$. Since each edge has length $\frac{\pi}{3}$, its girth is $6  \frac{\pi}{3}=2\pi$, so $X$ is CAT(1). A picture of the diagonal link of $NCP_4$ can be found in Figure~\ref{fig_ncp4}.
\end{rmk}

\begin{figure}[h]
\begin{center}
    	\resizebox{!}{5cm}{\includegraphics{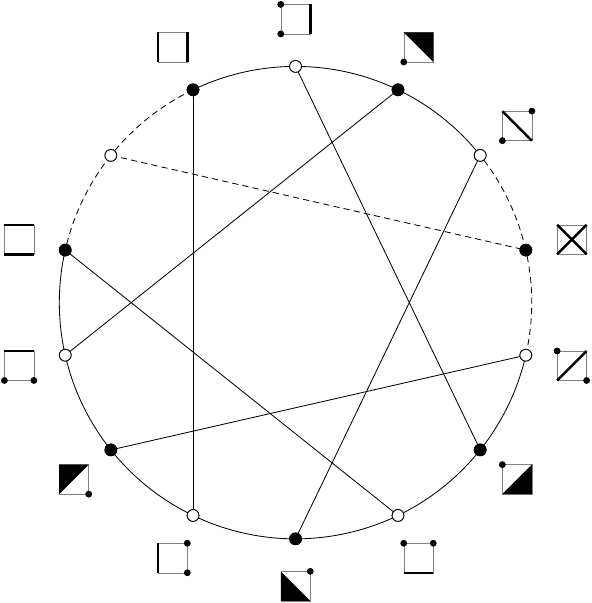}}
\caption[ncp4]{The diagonal link of $NCP_4$ is shown with solid lines. Dotted lines represent the missing two vertices of the Fano plane, one of which is a vertex of the diagonal link of the geometric realisation of the partition lattice. }
\label{fig_ncp4}
\end{center}
\end{figure}

\begin{dfn}[Non-crossing trees]
A \emph{non-crossing forest} of $U_n$ is a metric forest embedded in $\C$ with vertex set $U_n$, whose edges are geodesic segments in $\C$. When such a forest has only one connected component, we call it a \emph{non-crossing tree}
\end{dfn}

\begin{rmk}
\label{rmk: forests and vertices}
Note that every non-crossing forest corresponds to an element in $NCP_n$. The correspondence is obtained by saying that two points in $U_n$ lie in the same block \iff they lie in a single connected component of the forest.
In particular this gives a one-to-one correspondence between vertices of rank 1 in $X$ (that is the corresponding partition is of rank 1) and non-crossing forests with only one edge.

This way we can also  associate a subset of $X$ to a non-crossing tree by taking the span of all vertices associated to proper subforests of our non-crossing tree.
\end{rmk}

\begin{prop}
\label{prop: apartments and trees}
Let $A$ be an apartment in $B$. Then $A$ is included in $X$ if and only if its $(n-1)$ rank $1$ vertices lie in $X$ and correspond to the edges of a non-crossing tree.
\begin{proof}
Suppose $A$ is an apartment lying in $X$.
Each rank 1 vertex $v_i$ of $A$ corresponds to
a basis vector $\epsilon_i$ of $V$, the vector space that is used to define $B$ as the diagonal link in $S(V)$.
Each such vertex also corresponds to
an edge $e_i \subset \C$ connecting two points in $U_n$, as explained in the remark above. We claim that $T$, the union of edges $e_i$, is an embedded tree.

Let $v_i$ and $v_j$ be two distinct vertices of $A$ of rank 1. Then their join in $B$ has rank 2 (it is the plane $\langle \epsilon_i, \epsilon_j \rangle$), and lies in $A$. But $A \subseteq X$, and so the partition $v_i \vee v_j$ has rank 2. Observe that if $e_i$ intersects $e_j$ away from $U_n$, then the join $v_i \vee v_j$ has to contain the convex hull of $e_i \cup e_j$ as a block (since it is non-crossing), and therefore its rank is at least 3. Hence $e_i$ can intersect $e_j$ only at $U_n$, and therefore $T$ is embedded.

Now suppose that $T$ contains a cycle. Without loss of generality let us suppose that the shortest cycle is given by the concatenation of edges $e_1, \ldots, e_k$ for some $k$. Then note that the joins in $X$ satisfy
\[
\bigvee_{i=1}^k v_i = \bigvee_{i=1}^{k-1} v_i.
\]
But, as before, they are equal to the joins in $B$ (since $A \subseteq X$). This yields the equality
\[
\langle \epsilon_1, \ldots, \epsilon_k \rangle = \langle \epsilon_1, \ldots, \epsilon_{k-1} \rangle,
\]
which contradicts the fact that the vectors $\epsilon_i$ are linearly independent.

We have thus shown that $T$ is an embedded forest. But $T$ consists of ${n-1}$ edges, and so an Euler characteristic count yields that it has exactly one connected component. Therefore $T$ is a tree as required.

\smallskip

Now suppose the vertices of rank 1 of an apartment $A$ lie in $X$ and form a non-crossing tree $T$. The apartment is the span of the closure of the set of its rank 1 vertices under taking joins in $B$. The fact that $T$ is non-crossing tells us that the joins of these vertices taken in $B$ or $X$ coincide, and hence all vertices of $A$ lie in $X$. Therefore $A \subseteq X$.
\end{proof}
\end{prop}

\begin{dfn}[Universal points]
A point $x \in X$ is said to be \emph{universal} if it belongs to a face in $X$, all of whose vertices are partitions with exactly one block containing more than one element, and such that this block only contains consecutive elements of $U_n$. Such a face is also called \emph{universal}.
\end{dfn}

Note that every universal face is contained in a universal chamber.

\begin{ex}
Consider for $n=6$ the edge between the two partitions shown in Figure~\ref{fig_ab}, that is
\[
\{\{1,2,3\},\{4\},\{5\},\{6\}\}< \{\{1,2,3,4,5\},\{6\}\}.
\]
Then any point on the edge they form is universal in the sense just defined.
\end{ex}

\begin{figure}[h]
\begin{center}
    	\resizebox{!}{2cm}{\includegraphics{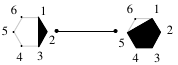}}
\caption[universal]{All points on this edge in $X$ are universal. }
\label{fig_ab}
\end{center}
\end{figure}

\begin{lem}
\label{lem: starring at univ}
Let $x \in X$ be a universal point, and let $y \in X$ be non opposite to $x$ in $B$, that is $d_B(x,y) < \pi$. Then the geodesic in $B$ between $x$ and $y$ lies in $X$. In other words, $X$ is $\pi$-star-shaped at $x$.
\begin{proof}
Let $C$ be a universal chamber with $x \in C$, and let $C'$ be a chamber containing $y$. We will construct an apartment $A \subseteq X$ containing both.

Let $\{x_1, \ldots, x_{n-2} \}$ denote the vertices of $C$, with indices corresponding to ranks. We are going to construct a total order on $U_n$. Note that, seen as partitions, $x_{i+1}$ is obtained from $x_i$ by expanding the unique block with multiple elements (which we will refer to a as the big block of $x_i$) by an element adjacent to the block. We will call this element the new element of $x_{i+1}$. We take our order to be one in which an element $v \in U_n$ is larger than $u$ whenever there exists $i$ such that $v$ is new for $x_{i+1}$, and $u$ belongs to the big block of $x_i$ (we allow $i=n-1$ and set $x_{n-1} = 1$, the partition with one element). Note that there are precisely two such total orders, depending on how we order vertices in the big block of $x_1$. Note also that given any element $v \in U_n$, we get a non-crossing partition $o_v$ with blocks
\[ \{ w \in U_n \mid w \leqslant v \} \textrm{ and } \{ w \in U_n \mid w > v \}.\]

Now let $\{y_1, \ldots, y_{n-2} \}$ denote the vertices of $C'$, with indices corresponding to ranks. Note that, seen as partitions, $y_{i+1}$ is obtained from $y_i$ by combining two blocks into one. We are now going to inductively construct embedded forests with vertex set $U_n$, and edges given by geodesic segments.

We set $T_1$ to be the forest with vertex set $U_n$, and a single edge connecting the two elements of the unique non-trivial block of $y_1$.
Suppose we have already defined $T_i$. Then $T_{i+1}$ is obtained from $T_i$ by adding an edge connecting elements $v$ and $w$ such that
\begin{itemize}
\item $v$ and $w$ do not lie in a common block in $y_i$;
\item $v$ and $w$ do lie in a common block in $y_{i+1}$;
\item $v>w$;
\item $v$ is the minimal element in its block in $y_i$; and
\item the new forest $T_{i+1}$ is embedded.
\end{itemize}
To show that such a pair $v,w$ exists let us look at minimal vertices in the two blocks of $y_i$ that become one in $y_{i+1}$. We let $v$ be the larger of the two. Then we know that the block not containing $v$ contains at least one smaller vertex. Together with the fact that $o_v$ defined above is non-crossing, the existence of a suitable $w$ is guaranteed.

Now it is clear that $T = T_{n-1}$ (with $y_{n-1} = 1$, the full partition) is an embedded tree with vertex set $U_n$. It is also clear that the apartment $A$ defined by $T$ (using Proposition~\ref{prop: apartments and trees}) contains $C'$.

Observe that every element (except the minimal one) is connected with an edge to a smaller element. This is due to the fact that every element except the minimal one stops being the smallest element in its block for some $i$ (when we add $y_{n-1} =1$ to our considerations). When it stops being minimal, it plays the role of $v$ above, and so is connected to a smaller element. From this we easily deduce that $x_i \in A$ for every $i$, and so that $C \subseteq A$.

Now both points $x$ and $y$ lie in a common apartment $A$, and the distance between them is smaller than $\pi$. Hence there exists a unique geodesic in $B$ between them, and it lies in $A$. But $A \subseteq X$, so this concludes the proof.
\end{proof}
\end{lem}

\begin{lem}
\label{lem: shrinking at univ}
Let $x \in X$ be a universal point, and let $l$ be a short loop in $X$ through $x$. Then $l$ is shrinkable in $X$.
\end{lem}

\begin{proof}
The Arzel\`a--Ascoli theorem tells us that we can assume without loss of generality that $l$ cannot be shrunk to a shorter loop going through $x$.
Then, since every point in $X$ has a neighbourhood isometric to a metric cone over the point, the loop $l$ is a locally geodesic path in $X$ except possibly at $x$. We claim that $l$ is constant.

By contradiction, assume that $l$ is not constant, and has length $0 < L < 2\pi$. View $l$ as a path $l:[0,L] \ra X$ from $x$ to $x$. As observed above, $l$ is a local geodesic. If $l$ is not locally geodesic in $B$, then it has a turning point in $(0,L)$. According to Lemma~\ref{lem: finite}, the set of turning points of $l$ is finite. Consider a turning point closest to $0$ or $L$; without loss of generality assume that $0< t \leq \frac{L}{2} < \pi$ is a turning point such that there is no turning point in $(0,t)$. Then $l|_{[0,t]}$ is a locally geodesic segment in $B$ of length smaller than $\pi$, hence it is a geodesic segment in $B$.

Then for $\varepsilon>0$ small, the geodesic segments $[x,l(t+\alpha \varepsilon)] \subset B$ for $\alpha \in (0,1]$  lie in $X$ by Lemma~\ref{lem: starring at univ}, and are shorter than $l|_{[0,t+\alpha \varepsilon]}$. They also vary continuously with $\alpha$, since $B$ is CAT(1) (compare Figure~\ref{fig_le3-28}). Therefore $l$ can be shrunk by replacing $l\vert_{[0,t+\alpha \varepsilon]}$ by $[x,l(t+\alpha \varepsilon)]$, which contradicts the assumption on $l$.

So $l$ is locally geodesic in $B$, and therefore $l|_{[0,\frac{L}{2}]}$ and $l^{-1}|_{[\frac{L}{2},L]}$ are two locally geodesic paths in $B$ from $x$ to $l(\frac{L}{2})$ of length smaller than $\pi$. They must be equal, since $B$ is CAT(1), and hence $l$ is constant.
\end{proof}

\begin{figure}[htbp]
\begin{center}
   	\resizebox{!}{3cm}{\includegraphics{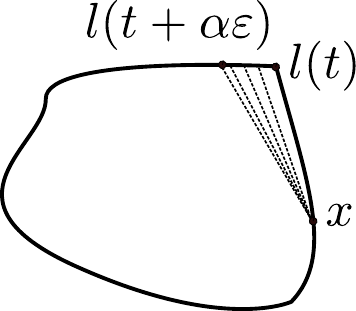}}
\caption[dominant]{Illustration of curve shortening in the proof of Lemma~\ref{lem: shrinking at univ}.}
\label{fig_le3-28}
\end{center}
\end{figure}

\begin{lem}[Failing modularity with respect to $NCP_n$]
\label{lem: modularity and convex hull}
Let $E, F$ be two faces of $X$. The simplicial convex hull of $E \cup F$ in $B$ is contained in $X$ \iff no two vertices of $E \cup F$ fail modularity with respect to $NCP_n$.
\end{lem}
\begin{proof}
In Lemma~\ref{lem:geodesic in building} we have already identified the simplicial convex hull $N$ of $E \cup F$ with the (full subcomplex spanned by) the smallest subset of the vertex set of $B$ stable under taking joins and meets, and containing the vertices of $E \cup F$. If a pair of vertices of $E \cup F$ fails modularity, then we immediately see that $N$ does not lie in $X$.

Let us now assume that no two vertices of $E \cup F$ fail modularity. This means that for vertices $x$ and $y$ in $E \cup F$ their joins and meets taken in $NCP_n$ agree with the joins and meets taken in $S(V)$.

Let us focus on joins for the moment; the situation for meets is analogous, and the results for joins are easily transferred to results for meets using the duality of $NCP_n$.
The join $x \vee y$ taken in $NCP_n$ equals the one taken in $S(V)$ \iff they are of the same rank. The rank can be easily read off the block structure of the partition, and thus we immediately see that the two joins agree \iff whenever
the convex hull of a block of $x$ intersects the convex hull of a block of $y$, then this intersection contains some point of $U_n$, i.e. no two blocks are crossing.

We claim that $N \subseteq X$, that is that we can perform sequences of meets and joins (in $S(V)$ on vertices of $E \cup F$ and never leave $NCP_n$. Let us suppose that this is not the case. Without loss of generality we may assume that there are vertices $z$ and $w$ in $X$, each obtainable from the vertices of $E \cup F$ by a sequence of meets and joins, and such that $z \vee w$ (taken in $S(V)$) does not lie in $X$. The discussion above tells us that $z$ and $w$ contain crossing blocks. In particular, there exist points $\alpha_1, \alpha_2, \beta_1, \beta_2 \in U_n$ such that $\alpha_1 \neq \alpha_2$ lie in a block of $z$, $\beta_1 \neq \beta_2$ lie in a block of $w$, and these two blocks cross.

Since $z$ is obtained from vertices of $E \cup F$ by taking joins and meets, there exists a vertex therein in which $\alpha_1$ and $\alpha_2$ are contained in a single block, and this block does not contain both $\beta_1$ and $\beta_2$ (it might however contain one of them). Without loss of generality let us assume that there exists a vertex of $E$ satisfying this property; let $e$ denote the minimal such vertex in $E$.

Let us suppose that there exists a vertex in $E$ satisfying the analogous property for $\beta_1$ and $\beta_2$; let $e'$ denote the minimal such vertex. Now if $e' \leq e$ then the block of $e$ containing $\alpha_1$ and $\alpha_2$ must also contain $\beta_1$ and $\beta_2$, since otherwise $e$ is not a non-crossing partition. This is a contradiction. Similarly, when $e < e'$, then $e'$ cannot be as defined.
We conclude that there exists a vertex in $F$ such that $\beta_1$ and $\beta_2$ lie in a common block thereof, and this block does not contain both $\alpha_1$ and $\alpha_2$. Let $f$ be the minimal such vertex. Note that $f$ is in fact the minimal vertex of $E \cup F$ satisfying the above property; using an analogous argument we show that $e$ is also the minimal vertex of $E \cup F$ in which $\alpha_1$ and $\alpha_2$ lie in a common block, which does not contain both $\beta_1$ and $\beta_2$.

We now claim that $e \vee f$ (taken in $S(V)$) does not lie in $NCP_n$.
It is enough to find a block in $e$ which crosses a block in $f$. We already have two candidate blocks, the one containing $\alpha_1$ and $\alpha_2$ in $e$ and the one containing $\beta_1$ and $\beta_2$ in $f$. It could happen however that these blocks have an element in $U_n$, say $\gamma$, in common. But then, using the minimality of $e$ and $f$, we conclude that the crossing blocks of $z$ and $w$ also contain $\gamma$, and thus are not crossing. This is a contradiction.

We have thus found two vertices in $E \cup F$ which fail modularity with respect to $NCP_n$.
\end{proof}

\begin{lem}
\label{lem: turning points univ for NCP5}
When $n=5$, turning faces in $X$ are universal vertices.
\begin{proof}
When $n=5$, by Lemma~\ref{lem: coconsecutive} we know that the corank of a turning face contains at least $2$ consecutive integers. Since the rank of $X$ is equal to $3$, we conclude that $F$ is a vertex of rank either $1$ or $3$. By Lemmata~\ref{lem: modular do not turn} and \ref{lem: modularity and convex hull}, we know that $F$ has two neighbours which fail modularity with respect to $NCP_5$, hence $F$ is necessarily (by inspection) a universal vertex.
\end{proof}
\end{lem}

\begin{cor}
\label{cor: NCP5 is CAT(1)}
The non-crossing partition complex $NCP_5$ is CAT(0).
\begin{proof}
Assume there is an unshrinkable short loop in $X$, the diagonal link of $NCP_5$. Then it has a turning point by Lemma~\ref{lem: cardinality of T}, which is a universal vertex by Lemma~\ref{lem: turning points univ for NCP5}, so the loop can be shrunk by Lemma~\ref{lem: shrinking at univ}. Hence by Theorem~\ref{thm: brady mccammond large} (and Remark~\ref{rmk: ncp 3 and 4}), $NCP_5$ is CAT(0).
\end{proof}
\end{cor}

\begin{figure}[htbp]
\begin{center}
   	\resizebox{!}{1cm}{\includegraphics{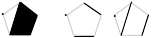}}
\caption[dominant]{On the left the turning vertex $\{\{1,2,3,4\},\{5\}\}$ in $NCP_5$, with two neighbours that fail modularity with respect to $NCP_5$. Compare Example~\ref{ex_modularity}}
\label{fig_modularity_nc}
\end{center}
\end{figure}

\begin{ex}[Vertices failing modularity with respect to $NCP_5$]\label{ex_modularity}
Figure~\ref{fig_modularity_nc} pictures the partition $\{\{1,2,3,4\},\{5\}\}$ in $NCP_5$, with a pair of neighbours given by the partitions $\{\{1,2\},\{3,4\},\{5\}\}$ and $\{\{2,3\},\{1,4\},\{5\}\}$. These neighbours fail modularity with respect to $NCP_5$ as defined in Definition~\ref{dfn: failing}. More explicitly this means the following: recall that the lattice of non-crossing partitions can be linearly embedded into a linear lattice $S(V)$. The partitions $\{\{1,2\},\{3,4\},\{5\}\}$ and $\{\{2,3\},\{1,4\},\{5\}\}$ then represent linear subspaces of the underlying vector space $V$. Their common span in the linear lattice is a linear subspace which cannot be represented (under our fixed embedding $NCP_n \to S(V)$) by a non-crossing partition.

The same vertex with crossing neighbours that fail modularity is illustrated in Figure~\ref{fig_modularity} where we also show how these vertices fit into the diagonal link of the partition complex.
Pictured are three apartments $A_1, A_2$ and $A_3$ of $B$. The apartments $A_1$ and $A_2$ are contained in $X$ and intersect in the gray-shaded region. The geodesic in $X$ connecting $u$ and $w$ runs via $v$. The apartment $A_3$ (which does not lie in $X$) contains $u,v$ and $w$ and of course also the (now shorter) geodesic in the diagonal link of the partition complex connecting $u$ and $w$. The intersection $A_1\cap A_3$ is shown in blue while the yellow area highlights  $A_2\cap A_3$.
\end{ex}

\begin{figure}[htbp]
\begin{center}
\vspace{2ex}
   	\resizebox{!}{6cm}{\includegraphics{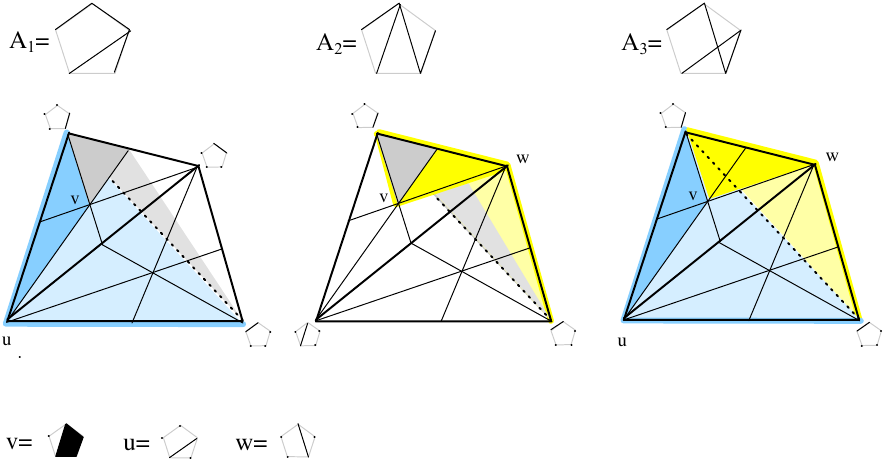}}
\caption[dominant]{This figure shows the vertex $v=\{\{1,2,3,4\},\{5\}\}$ in $NCP_5$ with two neighbours $u, w$ that fail modularity with respect to $NCP_5$. For more details see Example~\ref{ex_modularity}. }
\label{fig_modularity}
\end{center}
\end{figure}

\begin{dfn}[Dominant vertex]
A vertex $v$ of a face $F$ of $X$ is called \emph{dominant} \iff every apartment $A$ in $X$, with $v \in A$, contains $F$.
\end{dfn}

The following lemma will be accompanied by figures illustrating the cases and subcases.
In each case we look at a turning face which can either be one of the two edges illustrated, or a single vertex. Next to each turning edge (subcases (b) and (c)) we depict a pair of adjacent vertices that fail modularity. For the turning vertices (subcases (a)) at least one of the examples next to subcases (b) and (c) gives a pair of adjacent vertices which fail modularity.

The pictured partitions should be read as follows: the vertex in the upper left hand corner will be labeled by 1 and all the other vertices will be labeled clockwise from 2,...,6. Hence the partitions shown in case (1b) are (from left to right): \{\{1\}, \{2,3,4\}, \{5\}, \{6\}\}, connected to \{\{1\},\{2,4\}, \{3\}, \{5\}, \{6\}\}, then \{\{1,5\}, \{2,3,4\}, \{6\}\} and \{\{1\}, \{2,3,4,6\}, \{5\}\}.

\begin{lem}
\label{lem: turning points for NCP6}
When $n=6$, a non-universal turning face $F$ in $X$, up to the symmetries of $U_6$, falls into one of the following cases.
\begin{enumerate}
\item[(1a)] $F$ is the vertex $v = \{\{2,4\},\{1\},\{3\},\{5\},\{6\}\}$ of rank $1$
\item[(1b)] $F$ is the following edge of rank $(1,2)$ with dominant vertex $v$ as in (1a).
\begin{figure}[H]
\begin{center}
   	\resizebox{!}{1cm}{\includegraphics{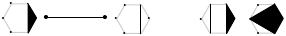}}
\end{center}
\end{figure}
\item[(1c)] $F$ is the following edge of rank $(1,4)$ with dominant vertex $v$ as in (1a).
\begin{figure}[H]
\begin{center}
   	\resizebox{!}{1cm}{\includegraphics{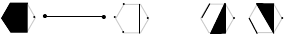}}
\end{center}
\end{figure}

\item[(2a)] $F$ is a single vertex of rank 2, namely either $v = \{\{1,2\},\{3,4\},\{5\},\{6\}\}$  or $v'= \{\{1,2\},\{4,5\},\{3\},\{6\}\}$
\item[(2b)] $F$ is the following edge of rank $(1,2)$ with dominant vertex $v'$ as in (2a)
\begin{figure}[H]
\begin{center}
   	\resizebox{!}{1cm}{\includegraphics{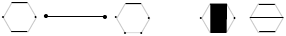}}
\end{center}
\end{figure}
\item[(2c)] $F$ is the following edge of rank $(1,2)$ with dominant vertex $v$ as in (2a)
\begin{figure}[H]
\begin{center}
   	\resizebox{!}{1cm}{\includegraphics{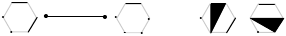}}
\end{center}
\end{figure}

\item[(3a)] $F$ is a single vertex of rank 3, namely either $v = \{\{1,2,3,5\},\{4\},\{6\}\}$ or $v'= \{\{1,2,4,5\},\{3\},\{6\}\}$
\item[(3b)] $F$ is the following edge of rank $(3,4)$ with dominant vertex $v$ as in (3a)
\begin{figure}[H]
\begin{center}
   	\resizebox{!}{1cm}{\includegraphics{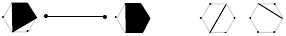}}
\end{center}
\end{figure}
\item[(3c)] $F$ is the following edge of rank $(3,4)$ with dominant vertex $v'$ as in (3a)
\begin{figure}[H]
\begin{center}
   	\resizebox{!}{1cm}{\includegraphics{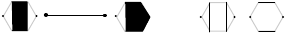}}
\end{center}
\end{figure}

\item[(4a)] $F$ is a single vertex $v = \{\{1,2,3,4\},\{5,6\}\}$ of rank $4$
\item[(4b)] $F$ is the following edge of rank $(3,4)$ with dominant vertex $v$ as in (4a)
\begin{figure}[H]
\begin{center}
   	\resizebox{!}{1cm}{\includegraphics{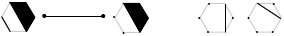}}
\end{center}
\end{figure}
\item[(4c)] $F$ is the following edge of rank $(1,4)$ with dominant vertex $v$ as in (4a)
\begin{figure}[H]
\begin{center}
   	\resizebox{!}{1cm}{\includegraphics{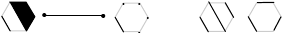}}
\end{center}
\end{figure}
\end{enumerate}
\end{lem}
\begin{proof}
The proof of this lemma is essentially an inspection.

When $n=6$, by Lemma~\ref{lem: coconsecutive} we know that the corank of a turning face contains at least $2$ consecutive integers. Since the rank of $X$ is equal to $4$, we conclude that $F$ is either a vertex or an edge. Using duality, we can restrict to the cases where $F$ is a vertex of rank $1$ or $2$, or an edge of rank $\{1,2\}$ or $\{1,4\}$. In these cases it is not hard to see that the cases listed are the only ones which allow for a pair of adjacent vertices failing modularity, in view of Lemmata~\ref{lem: modular do not turn} and \ref{lem: modularity and convex hull}.
Once we have listed all possible cases, we observe (again by inspection) that each turning face contains a dominant vertex.

Let us look at one example more closely; it is typical in the sense that in all the cases the argument is essentially the same.

Suppose that the turning face is a vertex of rank 1. Then, up to symmetry, it is either the vertex of subcase (1a), that is the partition $\{\{2,4\},\{1\},\{3\},\{5\},\{6\}\}$, or the vertex $\{\{2,5\},\{1\},\{3\},\{4\},\{6\}\}$. In the former case we can indeed find two vertices adjacent to our vertex which fail modularity. In the latter it is impossible: the lattice interval between that vertex $v$ and the maximal element $1$ is isomorphic to the product of two copies of the lattice $NCP_3$. As a consequence, no two vertices adjacent to $v$ can fail modularity, since the way $NCP_3$ embeds into $S(V)$ preserves meets and joins.
\end{proof}

When $F$ is a face in $X$ and $i \in U_n$, let $F_i \subseteq U_n$ denote the smallest subset of $U_n$ that appears as a block in a vertex of $F$ and that contains $i$ properly. If the set of such subsets is empty we set $F_i = U_n$.

\begin{lem}
\label{lem: leaf of tree}
Let $C$ be a chamber in $X$, and let $i, j$ be consecutive elements of $U_n$. If $C_i$ contains $j$, then there exists an apartment in $X$ containing $C$, $v$ and $w$, where $v$ is the universal vertex having the single nontrivial block $\{i,j\}$ and $w$ is the universal vertex opposite to $v$ in $B$ given by the partition $w=\{\{i\}, \llb 1,n \rrb \smallsetminus \{i\} \}$.
\begin{proof}
Write $v_1,\ldots , v_{n-2}$ for the vertices of $C$, with the indices corresponding to the ranks,  and let $v_{n-1}= 1$ be the maximal element in $NCP_n$.
Denote the edge $\{i,j\}$ by $e$.

Let $k$ be minimal such that $C_i$ is a block of $v_k$.

Any apartment in $X$ containing both $v$ and $w$ is represented by a non-crossing tree $T$ which contains the edge $e$, and such that the subforest obtained by removing $e$ from the tree corresponds to $w$ (since $w$ is opposite $v$), using the correspondence from Remark~\ref{rmk: forests and vertices}.
We will now construct such a tree $T$ by inductively picking edges $e_l=$ for $l=1,\dots, n-1$ in $U_n$.

 Take $e_1$ to be the edge corresponding to $v_1$. For each $2\leq l\leq k-1$ choose $e_l$ to be an edge such that the edges $e_1, \ldots, e_l$ form a non-crossing forest corresponding to $v_l$ (again using Remark~\ref{rmk: forests and vertices}). This is possible since vertices of $C$ are non-crossing partitions.

Choose $e_k=e$. Observe that the edges $e_1, \ldots, e_k$ still form a non-crossing forest, since the edge $e$ cannot cross any other edge, and the vertex $i$ was isolated in the forest formed by $e_1, \ldots, e_{k-1}$, and so no cycles appear.

Now we continue choosing edges $e_l$ for $k+1, \ldots, n-1$ as before, with the additional requirement that
none of the edges $e_l$ with $l\geq k+1$  connects to $i$. Choosing the remaining edges like this is possible since in each step two blocks of $v_{l}$ are joined to form a block of $v_{l+1}$, and the block containing $i$ always contains at least also the vertex $j$, hence if a block is joined to the one containing $i$ then we may do this using an edge emanating from $j$ (or some other vertex in this block different from $i$).
The resulting tree is by construction non-crossing. The apartment $A$ spanned by $T$ contains $v$ and $C$. Further, since $e_k$ is the only edge connected to $i$, the apartment $A$ does also contain the vertex $w$.
\end{proof}
\end{lem}

\begin{lem}
\label{lem: shortcut for NCP6}
When $n=6$, let $F$ be a non-universal turning face in $X$, and let $C$ be any chamber in $X$. Then there exists a pair $v,w$ of universal vertices in $X$, which are opposite in $B$, such that $F,v,w$ are contained in an apartment in $X$, and $C,v,w$ are contained in a (possibly different) apartment in $X$.
\begin{proof}
According to Lemma~\ref{lem: turning points for NCP6}, every non-universal turning face contains a dominant vertex; let $u$ denote the dominant vertex of $F$. Our strategy here is to find consecutive $i, j\in U_n$ such that $j \in C_i \cap u_i$. Then Lemma~\ref{lem: leaf of tree} will give a pair $v,w$ of universal vertices in $X$, which are opposite in $B$, such that there exists an apartment in $X$ containing $C,v,w$, and another apartment in $X$ containing $u,v,w$. Since $u$ is dominant for $F$, this last apartment contains $F,v,w$.

The following table lists all possibilities for $u_i$ (up to duality), depending on $i$ and the dominant vertex $u$ of $F$ (listed as in Lemma~\ref{lem: turning points for NCP6}).
\[\begin{array}{c|cccccc}
\textrm{Dominant vertex $u$}& i=1& 2& 3& 4& 5& 6 \\
 \hline \{ \{2,4\}, \{1\},\{3\}, \{5\}, \{6\}\} & U_6 & \{2,4\} & U_6 & \{2,4\} & U_6 & U_6\\
\{ \{1,2\}, \{3,4\}, \{5\}, \{6\}\} & \{1,2\} & \{1,2\} & \{3,4\} & \{3,4\} & U_6 & U_6\\
\{ \{1,2\}, \{4,5\}, \{3\}, \{6\}\} & \{1,2\} & \{1,2\} & U_6 & \{4,5\} & \{4,5\} & U_6\\
\end{array}\]

Now let us consider $C_i$. If $5 \in C_6$ or $1 \in C_6$ then our table tells us that we are done. Suppose that neither of these two occurs. If $4 \in C_6$ then $4 \in C_5$ and again we are done. Similarly if $2 \in C_6$ then $2 \in C_1$.

We are left with the case $C_6 = \{3,6\}$.
Here $6 \in C_5$ or $4 \in C_5$, which deals with the first two possibilities for $u_5$. The third one requires the observation that if $4 \not\in C_5$ then $5 \in C_4$.
\end{proof}
\end{lem}

\begin{thm}
\label{thm: NCP6 is CAT(1)}
The non-crossing partition complex $NCP_6$ is CAT(0).
\begin{proof}
Assume there is an unshrinkable short loop $l:S^1 \ra X$ of length $L < 2\pi$, where $X$ is the diagonal link of $\vert NCP_6\vert$. Then by Lemma~\ref{lem: cardinality of T} this loop has a turning point with image $x$ in $X$. Let us reparametrise $l$ so that the domain of $l$ is $[0, L]$ and $x=l(0)=l(L)$. Consider $y=l(L/2)$. By Lemma~\ref{lem: shortcut for NCP6}, there exists a pair $v,w$ of universal vertices in $X$, which are opposite in $B$, such that both $\{x,v,w\}$ and $\{y,v,w\}$ lie in apartments in $X$. Hence we know that
\[d(x,v)+d(x,w)=d(y,v)+d(y,w)=\pi.\]
So at least one element of $\{v,w\}$, say $v$, satisfies \[d(x,v)+d(v,y) \leq \pi.\]
Let $\alpha_1 = l|_{[0,L/2]}$ and $\alpha_2 = l|_{[L/2,L]}^{-1}$ be the two subpaths of $l$ from $x$ to $y$. Let $\alpha_3 \colon [0,d(x,v)+d(v,y)] \ra X$ denote the concatenation of the geodesic segments $[x,v]$ and $[v,y]$.

Consider the loop $\alpha_3^{-1} \circ \alpha_1$. Since it is short and passes through the universal vertex $v$, by Lemma~\ref{lem: shrinking at univ} it can be shrunk. Similarly, the loop $\alpha_3^{-1} \circ \alpha_2$ can be shrunk. Now we can apply Theorem~\ref{thm: bowditch shortcut}, which tells us that the loop $l=\alpha_2^{-1} \circ \alpha_1$ can be shrunk. Hence by Theorem~\ref{thm: brady mccammond large}, the diagonal link in $NCP_6$ is CAT(1), and the result follows.
\end{proof}
\end{thm}

Now we apply Proposition~\ref{prop: cat0 implies braids} to conclude the following.

\begin{cor}
\label{cor: braid groups are CAT0}
For every $n \leq 6$, the $n$-strand braid group is CAT(0).
\end{cor}

\section{The orthoscheme complex of a modular complemented lattice is CAT(0)}

We now prove that the orthoscheme complex of a bounded graded modular complemented lattice is CAT(0), thus giving a partial confirmation of~\cite[Conjecture~6.10]{brady_mccammond}. The conjecture states that the result should be true without assuming that the lattice is complemented, however we need this extra assumption to embed the diagonal link of the orthoscheme complex into a spherical building (or a CAT(1) graph in a pathological case).

\begin{dfn}[Modular lattice]
\label{dfn: modular}
 A lattice $P$ is said to be \emph{modular} if $$\forall x,y,z \in P,\; x \geq z \Longrightarrow x \wedge (y \vee z) = (x \wedge y) \vee z.$$
\end{dfn}

Suppose $P$ is a modular lattice which is linearly embedded in some $S(V)$. Then it is easy to check that joins and meets in $P$ need to coincide with joins, respectively meets in $S(V)$. Hence there is no pair $x,y\in P$ which fails modularity in the sense of Definition~\ref{dfn: failing}.

\begin{dfn}[Complemented lattice]
\label{dfn: complemented}
 A bounded lattice $P$ is said to be \emph{complemented} if $$\forall x \in P, \exists y \in P,\; x \wedge y = 0 \mbox{ and } x \vee y = 1.$$
\end{dfn}

\begin{dfn}[Plane lattice]
\label{dfn: plane lattice}
 A lattice $P$ is said to be a \emph{plane lattice} if it is bounded, and graded of rank $3$.
\end{dfn}

\begin{thm}[Frink's embedding Theorem]
\label{thm:modular complemented in building}
Let $P$ be a bounded graded modular complemented lattice. Then $P$ is isomorphic to a direct product
\[P = \prod_{i=1}^r P_i\]
of bounded graded modular complemented lattices, such that for all $i \in \llb 1,r \rrb$, the lattice $P_i$ can be embedded as a subposet of a linear lattice (over a division algebra) or of a plane lattice, where the embedding preserves the meets and the joins.
\end{thm}
Note that it is absolutely crucial for us that the embeddings preserve joins and meets. It immediately implies that if the diagonal link $X$ in $\vert P \vert$ is linearly embedded, then no two vertices in $X$ fail modularity with respect to $P$.
\begin{proof}
According to~\cite[Theorem~279]{gratzer} and \cite[Lemma~99]{gratzer}, $P$ is isomorphic to a direct product $P = \prod_{i=1}^r P_i$ of simple bounded graded modular complemented lattices, where the fact that each $P_i$ is simple implies that $P_i$ cannot be embedded non-trivially as a subposet of a non-trivial product of lattices. According to~\cite[Corollary 439]{gratzer}, each $P_i$ is then embedded as a subposet of a product of linear lattices (over a division algebra) or of plane lattices, such that the joins and the meets are preserved. Since $P_i$ is simple, the product consists of only one non-trivial factor.
\end{proof}

For a more precise version of Frink's embedding Theorem, we refer the reader to~\cite{gratzer}.

\begin{cor}
The diagonal link in the orthoscheme complex of a plane lattice is CAT(1).
\end{cor}
\begin{proof}
The diagonal link of the orthoscheme complex of a plane lattice is a graph, since the orthoscheme complex of a rank 3 poset has dimension 3, and so the diagonal link has dimension 1. Moreover, any cycle in the graph is of even length, since the lattice is graded, there are no 2-cycles, since it is a simplicial complex, and no 4-cycles, since the poset is a lattice. Thus the girth of the diagonal link is at least 6. Also, each edge has the same length, namely $\frac \pi 3$. Thus all loops shorter than $2 \pi$ are shrinkable. The graph is also clearly locally CAT(1).
\end{proof}

\begin{thm}
\label{thm:modular are cat0}
The orthoscheme complex of a bounded graded modular complemented lattice is CAT(0).
\begin{proof}
Let $P$ be a bounded graded modular complemented lattice, and let $|P|$ be its orthoscheme complex. By Theorem~\ref{thm:modular complemented in building}, write $P = \prod_{i=1}^r P_i$. Since the orthoscheme complex of $|P|$ is the Euclidean product of the orthoscheme complexes of the posets $P_i$ (thanks to~\cite[Remark 5.3]{brady_mccammond}), we only need to show that each $|P_i|$ is CAT(0).

According to~\cite[Theorem~5.10]{brady_mccammond}, it is enough to check that the diagonal links of the full subcomplexes of $\vert P_i \vert$ spanned by intervals in $P_i$ are CAT(1). Since every such interval is itself a bounded graded modular complemented lattice by~\cite[Lemma~98]{gratzer}, and since this subcomplex is isometric to the orthoscheme complex of the interval, we only need to check this property for the diagonal link of $\vert P_i \vert$ itself (formally, we proceed by induction on the rank of the lattice).

Fix $i \in \llb 1,r\rrb$. The lattice $P_i$ is embedded as a subposet of a linear lattice $S(V)$ (over a division algebra) or of a plane lattice $L$, where the embedding preserves the meets and the joins.

In the first case, $P_i$ is linearly embedded in $S(V)$ in the sense of Definition~\ref{dfn: linearly embedded}. Since joins and meets coincide in $P_i$ and in $S(V)$, we deduce that $P_i$ has no pair of elements failing modularity in the sense of Definition~\ref{dfn: failing}. Let
\[X=LK(e_{01},|P_i|) \subseteq LK(e_{01},|S(V)|)=B\]
 be their diagonal links. Since $B$ is a spherical building, it is CAT(1). Assume that there is a short locally geodesic loop $l$ in $X$. By Lemma~\ref{lem: cardinality of T}, the loop $l$ has a turning point in $B$.
By Lemma~\ref{lem: modular do not turn}, the image of that turning point has neighbours which fail modularity, which is a contradiction. Hence $l$ cannot exist, and so according to~\cite[Theorem~5.10]{brady_mccammond} this implies that $|P_i|$ is CAT(0).

In the second case, the diagonal link of $|P_i|$ is a subgraph of the diagonal link of a plane lattice, which is a CAT(1) graph. Thus so is the diagonal link of $\vert P_i \vert$, and thus $\vert P_i \vert$ is CAT(0) as before.
\end{proof}
\end{thm}

\bibliographystyle{abbrv}
\bibliography{bibliography}

\end{document}